\documentclass[review]{elsarticle}

\usepackage{hyperref}
\usepackage{amsfonts}
\usepackage{bbm}
\usepackage{graphicx}
\usepackage{float}
\usepackage{amsmath,amscd,amsbsy,amssymb,latexsym,url,bm,amsthm}
\usepackage[a4paper, left=2.8cm, right=2.8cm, top=2cm, bottom=2cm]{geometry}
\usepackage[vlined,boxed,commentsnumbered,ruled]{algorithm2e}
\usepackage{epsfig,graphicx,subfigure}
\usepackage{enumitem,balance,mathtools}
\usepackage{wrapfig}
\usepackage{mathrsfs, euscript}
\usepackage[usenames]{xcolor}
\usepackage{caption}
\usepackage{lipsum}
\usepackage{siunitx}		
\usepackage{bibentry}

\newtheorem{thm}{Theorem}[section]
\newtheorem{coro}[thm]{Corollary}
\newtheorem{lemma}[thm]{Lemma}
\newtheorem{prop}[thm]{Proposition}
\newtheorem{defn}[thm]{Definition}

\newtheorem{rem}[thm]{Remark}

\numberwithin{equation}{section}\allowdisplaybreaks

\makeatletter \renewenvironment{proof}[1][Proof]
{\par\pushQED{\qed}\normalfont\topsep6\p@\@plus6\p@\relax\trivlist\item[\hskip\labelsep\bfseries#1\@addpunct{.}]\ignorespaces}{\popQED\endtrivlist\@endpefalse} \makeatother
\makeatletter

\def\leq{\leqslant}

\def\eps{\varepsilon}
\def\leq{\leqslant}
\def\geq{\geqslant}
\def\real{{\mathbb{R}}}			
\def\FF{{\mathscr{F}^{-1}}}		
\def\F{{\mathscr{F}}}			
\def\sch{\mathscr{S}}

\def\mpq{M_{p,q}}

\newcommand{\norm}[1]{\left\|#1\right\|}		
\newcommand{\abs}[1]{\left|#1\right|}			
\newcommand{\set}[1]{\left\{#1\right\}}			
\newcommand{\inner}[1]{\left\langle#1\right\rangle} 
\newcommand{\rmnum}[1]{\uppercase\expandafter{\romannumeral#1}}  

\newcommand{\Cas}[1]{\begin{cases*} #1 \end{cases*}}
\newcommand{\sett}[1]{\left\{#1\right\}}			
					
\DeclareMathOperator*{\esssup}{ess\,sup}

\def\brk#1{\left(#1\right)}							
\def\sch{\mathscr{S}}
\def\eps{\varepsilon}
\def\leq{\leqslant}
\def\geq{\geqslant}
\def\mpq{M_{p,q}}					
\def\mpqs{M_{p,q}^{s}}

\def\real{{\mathbb{R}}}			
\def\FF{{\mathscr{F}^{-1}}}		
\def\F{{\mathscr{F}}}

\begin{document}

\begin{frontmatter}

\title{Global well-posedness for generalized fractional Hartree equations with rough initial data in all dimensions}

\author[address1]{Yufeng Lu}
\ead{lyf@jmu.edu.cn}

\address[address1]{School of Science, Jimei University, Xiamen, 361021, PR China}

\begin{abstract}
We prove the global existence of the solution for fractional Hartree equations with initial data in certain  real interpolation spaces between $L^{2}$ and some kinds of new function spaces defined by fractional Schr\"odinger semigroup, which could imply the global well-posedness of the equation in modulation spaces $M_{p,p'}^{s_{p}}$ for $p$ close to 2 with no smallness condition on initial data, where $s_{p}=(m-2)(1/2-1/p)$.  The proof adapts a splitting method inspired by the work of Hyakuna-Tsutsumi, Chaichenets et al. to the modulation spaces and exploits polynomial growth of the fractional Schr\"odinger semi-group on modulation spaces $M_{p,p'}$ with loss of regularity $s_{p}$.
\end{abstract}

\begin{keyword}
Fractional Hartree equation; Global well-posedness; Modulation space; Real interpolation\\
MSC Classification: 35R11, 35Q55, 35Q60, 42B37
\end{keyword}

\end{frontmatter}


\section{Introduction}
Let $m\geq 2$, we consider the following  fractional nonlinear Schr\"odinger equation with convolution nonlinearity:
\begin{align}
	\begin{cases} \tag{fNLS-h}\label{eq-fnlsh}
		i \partial_{t} u +(-\Delta)^{m/2} u = \pm \brk{K*\abs{u}^{\beta}} u,\\
		u(0,x) = u_{0}(x),
	\end{cases}
\end{align}
where $(t,x) \in \real^{+}\times \real^{d}, u(t,x)\in \mathbb{C},$  $K$ denotes the Hartree kernel $ K(x)=\dfrac{1}{\abs{x}^{\gamma}},$ for $x\in \real^{d}, 0<\gamma<d,$ and $*$ denotes the convolution in $\real^{d}$. The fractional Laplacian is defined as $$(-\Delta)^{m/2} u=\FF \abs{\xi}^{m}\F u(\xi),$$ where $(\FF) \F$ denotes the (inverse) Fourier transform in $\real^{d}$. 

The equation \eqref{eq-fnlsh} is known as the Hartree type equation with fractional Laplacian. We also call it fractional Hartree equation. This equation is a generalization of the classical Hartree equation with $\beta=m=2$, which commonly referred to as Boson star equation to be regarded as the classical limit of a field equation that describes a non-relativistic many-boson system interacting through a two-body potential (see \cite{Froehlich2004Mean,Ginibre1980class}). 
 
The modulation space $\mpq^{s}$ is the function space, introduced by Feichtinger \cite{Feichtinger2003Modulation} in the 1980s using the short-time Fourier transform to measure the decay and the regularity of the function differently from the usual $L^{p}$ Sobolev spaces or Besov-Triebel spaces.  For example, the Schr\"odinger and wave semigroups, which are not bounded on neither $L^{p}$ nor $B_{p,q}^{s}$ for $p\neq 2$, are bounded on $\mpq^{s}$  (see \cite{Benyi2007Unimodular}). Therefore, the modulation space is a good space for initial data of the Cauchy problem for nonlinear dispersive equations.

Denote 
\begin{align*}
	S_{m}(t) & :=\FF e^{it\abs{\xi}^{m}} \F;\\
	\mathscr{A} f(t) & := \int_{0}^{t} S_{m}(t-\tau) f(\tau) d\tau.
\end{align*}
We call $u$ is a solution of \eqref{eq-fnlsh} if it satisfies the integration form of \eqref{eq-fnlsh} as follows.
\begin{align*}
	u=S_{m}(t)u_{0}\mp i\mathscr{A} \brk{K*\abs{u}^{\beta}} u.
\end{align*}
When in its lifespan, the solutions to \eqref{eq-fnlsh} satisfy mass conservalation:
\begin{align*}
	M[u(t)] &:=\int_{\real^{d}} \abs{u(t,x)}^{2}dx =M[u_{0}].
\end{align*}

Moreover, if $u$ is a solution to \eqref{eq-fnlsh}, so is the scaled function 
\begin{align*}
	u_{\lambda}(t,x) & =\lambda^{\frac{m+d-\gamma}{\beta}}u(\lambda^{m}t,\lambda x).
\end{align*}
The critical exponent $s_{c}$ is the unique real number conserving the homogeneous Sobolev norm 
\begin{align*}
	\norm{u_{\lambda}(0,x)}_{\dot{H^{s_{c}}}} =\norm{u(0,x)}_{\dot{H^{s_{c}}}}, \ s_{c}=\frac{d}{2}-\frac{m+d-\gamma}{\beta}.
\end{align*}

 We first recall the notion of well-posendess in the sense of Hadamard.
\begin{defn}
	Let $X,Y\hookrightarrow \sch'$ be Banach spaces. The Cauchy problem \eqref{eq-fnlsh} is said to be well-posed from $X$ to $Y$, if for each bounded subset $B\subseteq X$, there exist $T>0$ and a Banach subspace $X_{T}\subseteq C([0,T],Y)$ such that 
	\begin{enumerate}
		\item for all $u_{0}\in B$, \eqref{eq-fnlsh} has a unique solution $u\in X_{T}$ with $u(0,x)=u_{0}(x)$;
		\item the mapping $u_{0} \rightarrow u$ is continuous from $(B,\norm{\ \cdot\ }_{X})$ to $C([0,T],Y)$.
	\end{enumerate}
	In particular, when $X=Y$, we say $\eqref{eq-fnlsh}$ is locally well-posed in $X$. And further, if $T$ can be chosen arbitrarily, then we say \eqref{eq-fnlsh} is globally well-posed in $X$.
\end{defn}

When $m=\beta=2$, \eqref{eq-fnlsh} is just the standard cubic nonlinear Schr\"odinger equation with Hartree type nonlinearity(NLS-h), which has been studied a lot.  There are also some well-posedness results for general cases of $m$ and $\beta$. Here, we mainly list the results in modulation spaces and Fourier Lebesgue spaces. For results in classical Sobolev spaces $H^{s}$, one can refer to \cite{Holmer2010Asymptotically,Miao2015Dynamics,Guo2018Sharp,Wang2023blow,Hayashi2024Higher,Zhou2025Threshold,Li2025Dynamics}.

When $0<\gamma<\min\set{2,d/2}$,
  Carles and Mouzaoui  \cite{Carles2014Cauchy} obtained the global wellposedness of (NLS-h) with initial data in certain Wiener algebra($L^{2}\cap\F L^{1}$).
Bhimani \cite{Bhimani2016Cauchy} obtained that there exists a unique global solution of (NLS-h) in $C(\real,M_{1,1})$. Manna \cite{Manna2017Modulation,Manna2019Cauchy} extended this result with initial data in $M_{p,p}, \ 1\leq p<\frac{2d}{d+\gamma}$. When $d=1$, Manna \cite{Manna2022existence} obtained the global well-posedness of (NLS-h) with initial data in $M_{p,p'}$ when $p>2$ is sufficiently close to 2. When $0<\gamma<\min\set{2,d}$, Hyakuna \cite{Hyakuna2018Multilinear,Hyakuna2019global} obtained the global solution of (NLS-h) for large data in $\F L^{r}$ or $L^{p}$, where $r>2,p<2$ are sufficiently close to 2. Hyakuna \cite{Hyakuna2019globala} also obtained the global solution with some twisted persistence property exists for data in the space $L^{1}\cap L^{2}$ under some suitable conditions on $\gamma$ and $d$.

When $\beta=2 \text{ and }m\leq2$, Bhimani \cite{Bhimani2019Global} obtained the global solution of \eqref{eq-fnlsh} with radial initial data in $\mpq$ for $1\leq p\leq 2, 1\leq q< 2d/(d+\gamma)$, when $d\geq 2, m>2d/(2d-1)$. Recently, Bhimani and his cooperators \cite{Bhimani2024Fractional} also obtained the global well-posedness  of \eqref{eq-fnlsh} with radial initial data in $M_{p,p'}$, when $p$ is sufficiently close to 2.

Bhimani and Haque \cite{Bhimani2023Strong} obtained the ill-posedness of \eqref{eq-fnlsh} in Fourier amalgam spaces with negative regularity for all $m>0$. These amalgam spaces are, under certain special circumstances, modulation spaces or  Fourier Lebesgue spaces. For the well-posedness results in such spaces, one can  refer to \cite{Bhimani2025mixed}.

In this paper, we focus on the \eqref{eq-fnlsh}, where the critical exponent satisfies $-\dfrac{md}{d-\gamma}<s_{c}<0$, which is equivalent to 
\begin{align} \label{eq-beta-condition}
	2(1-\frac{\gamma}{d})<\beta<2(1-\frac{\gamma}{d})+\frac{2m}{d}.
\end{align} 
The condition $s_{c}<0$ is also referred to as subcritical mass. Generally, in this case, by combining the law of conservation of mass, we can obtain the global well-posedness of the equation \eqref{eq-fnlsh} in $L^{2}$ (see Lemma \ref{lem-LWP-L2}). The main goal of this paper is to obtain the global well-posedness of the equation \eqref{eq-fnlsh} in a broader class of spaces, including certain modulation spaces and Fourier Lebesgue spaces.

We  exploit this splitting method in \cite{Chaichenets2017existence} to \eqref{eq-fnlsh} to obtain the global well-posedness in the real interpolation of $L^{2}$ and some kinds of new spaces defined by fractional Schr\"odinger semigroup. This is because  that the real interpolation could be defined as a subspace of the sum of these two spaces by $K$-method, which is harmonic with the splitting method.

We first recall the definition of real interpolation spaces.
\begin{defn}[Real interpolation; Chapter 3, \cite{Bergh1976Interpolation}]\label{def-real-interpolation}
	Let $\overline{A}=(A_{0},A_{1})$ be a Banach couple. For any $a\in A_{0}+A_{1}, t>0$, denote \begin{align*}
		K(t,a)= \inf_{a=a_{0}+a_{1}} \left(\norm{a_{0}}_{A_{0}}+t\norm{a_{1}}_{A_{1}}\right).
	\end{align*}
	For any $0\leq \theta \leq 1$, define the real interpolation space $\overline{A}_{\theta,\infty}=(A_{0},A_{1})_{\theta,\infty}$:
	\begin{align*}
		\overline{A}_{ \theta,\infty}= \sett{a\in  A_{0}+A_{1}: \sup_{t>0} t^{-\theta} K(t,a) <\infty}.
	\end{align*}
\end{defn}

We use the fractional Schr\"odinger semigroup $S_{m}(t)$ to define a new kind of function spaces, denoted by $X_{q,r}^{s,m}$.  For $T>0$, we denote $I_{T}=[0,T], \inner{T}=\brk{1+T^{2}}^{1/2}$. For $a,b\in \real$, we denote $a\vee b =\max\sett{a,b}, a\wedge b= \min \sett{a,b}$.

\begin{defn}\label{def-Xqr-sigma}
	Let $1\leq q,r\leq \infty, s \in \real$. Define the $X_{q,r}^{s,m}$ norm:
	\begin{align*}
		\norm{u}_{X_{q,r}^{s,m}}:= \sup_{T>0} \inner{T}^{s} \norm{S_{m}(t+T) u}_{L_{I_{1}}^{q} L_{x}^{r}},
	\end{align*}
	and denote $X_{q,r}^{s,m}$ the closure of the Schwartz function $\sch$ with finite $X_{q,r}^{s,m}$ norm.
\end{defn}

\begin{rem}
	$X_{q,r}^{s,m}$ is not trivial. In the proof of Corollary \ref{cor-Mpp'}, we will obtain that $M_{r,r'}^{s_{r}} \hookrightarrow X_{q,r}^{1/r-1/2,m} $ for any $1\leq q,2\leq r$.
\end{rem}

Denote 
\begin{align*}
	&\eta :=\frac{\gamma}{d};\\
	&\tilde{\alpha}(q,r):= \frac{1}{q}+\frac{\sigma}{r};\\
	&\frac{1}{\gamma_{m}(r)}:=\sigma(\frac{1}{2}-\frac{1}{r});\\
	&\tilde{\beta}:=\beta-2(1-\eta); \\ 
	&\beta_{0}:=\frac{3-2\eta}{\sigma-1};\\
	& a(\eta,\sigma):=\frac{\sigma-1}{2\sigma} \vee \frac{\frac{3}{2}-\eta}{\beta+1};\\
	& b(\eta,\sigma):=\frac{1}{2}\wedge \frac{\frac{1}{2\sigma}+\frac{3}{2}-\eta}{\beta+1}.
\end{align*}
and the closed interval 
\begin{align*}
	I_{\eta,\sigma}:= \left[a(\eta,\sigma),b(\eta,\sigma)\right]= 
	\Cas{
		\left[\frac{\frac{3}{2}-\eta}{\beta+1}, \frac{1}{2}\right], & if $(2\eta-1)\sigma \leq 2, \sigma\geq 1, 2\eta-2 \leq \tilde{\beta}\leq \frac{1}{\sigma}$;\\
		\left[\frac{\frac{3}{2}-\eta}{\beta+1}, \frac{\frac{1}{2\sigma}+\frac{3}{2}-\eta}{\beta+1}\right], & if $(2\eta-1)\sigma\leq 2, \sigma\geq 1, \frac{1}{\sigma} <\tilde{\beta}<\frac{2}{\sigma}$;\\
		\left[\frac{\frac{3}{2}-\eta}{\beta+1}, \frac{1}{2}\right],& if $(2\eta-1)\sigma \leq 2, \sigma< 1, 2\eta-2 \leq \tilde{\beta}\leq 1$;\\
		\left[\frac{\frac{3}{2}-\eta}{\beta+1}, \frac{2-\eta}{\beta+1}\right],& if $(2\eta-1)\sigma \leq 2, \sigma< 1, 1< \tilde{\beta}\leq \frac{2}{\sigma}$;\\
		\left[\frac{\frac{3}{2}-\eta}{\beta+1}, \frac{1}{2}\right], &  if $(2\eta-1)\sigma > 2,  2\eta-2 < \tilde{\beta}\leq \frac{1}{\sigma}$;\\
		\left[\frac{\frac{3}{2}-\eta}{\beta+1}, \frac{\frac{1}{2\sigma}+\frac{3}{2}-\eta}{\beta+1}\right], & if $(2\eta-1)\sigma> 2,  \frac{1}{\sigma} <\tilde{\beta}<\beta_{0}$;\\
		\left[\frac{\sigma-1}{2\sigma}, \frac{\frac{1}{2\sigma}+\frac{3}{2}-\eta}{\beta+1}\right], & if $(2\eta-1)\sigma> 2,  \beta_{0} <\tilde{\beta}<\frac{2}{\sigma}$.
}
\end{align*}

\begin{thm}[LWP] \label{thm-lwp}
	Let 
	\begin{align*}
		\beta<\frac{2}{\sigma}+2-2\eta, \  \frac{1}{r}\in I_{\eta,\sigma}, \  \tilde{\alpha}(q,r)<\frac{1+\sigma(\frac{3}{2}-\eta)}{\beta+1}.
	\end{align*}
	Then  \eqref{eq-fnlsh} is local wellposed in $X_{q,r}^{s,m}+L^{2}$, which means that for any initial data $u_{0}\in X_{q,r}^{s,m}+L^{2}$, there exists $T>0$ and a unique solution $u\in C(I_{T}, X_{q,r}^{s,m}+L^{2}) \cap L_{I_{T}}^{A}L_{x}^{r}$, where $A=q\wedge \gamma_{m}(r)$. Moreover, the data-to-solution map above is Lipschitz continuous.
\end{thm}

Denote  
\begin{align*}
	&\sigma=\frac{d}{m}, \quad \qquad \ y_{0}=\frac{\beta d-2(d-\gamma)}{2m(\beta+1)}, \quad \qquad y_{1}=\frac{d(\beta-2)+2\gamma}{m(\beta+2)},\\
	&x_{0}:=\frac{\frac{3}{2}-\eta}{\beta+1},\quad x_{1}:=\frac{2-\eta}{\beta+2}, \qquad \qquad \qquad\,\, x_{2}:=\frac{1-\eta}{\beta},\\
	&x_{3}:=\frac{2-\eta}{\beta+1}, \quad x_{4}:=\frac{(\sigma-1)(2-\eta)}{(\beta+2)(\sigma-1)+1}, \quad x_{5}:=\frac{\sigma(2-\eta)}{(\beta+2)\sigma-1}, 
\end{align*}
and 
\begin{align*}
&\Omega_{\gamma}:=\set{(x,y)\in \real^{2}: 0<y<2\sigma(\frac{1}{2}-x),y<\sigma(\beta x-1+\eta),y>\sigma\left(2-.-(\beta+2)x\right)};\\
&\Omega_{\gamma,\sigma}:=\Cas{\Omega_{\gamma}\cap\set{(x,y)\in \real^{2}:0<x<x_{3}}, &if $\sigma\leq1$;\\
\Omega_{\gamma}\cap \set{(x,y)\in \real^{2}: 0<x<x_{3}, x_{4}<x<x_{5}}, &if $\sigma>1$.}
\end{align*}
The area indicated by $\Omega_{\gamma}$ can be seen in Figure \ref{fig:omegagamma}. The condition $\beta>2(1-\eta)$ in \eqref{eq-beta-condition} could guarantee that $\Omega_{\gamma}$ is not empty.

\begin{figure}
	\centering
	\includegraphics[width=0.7\linewidth]{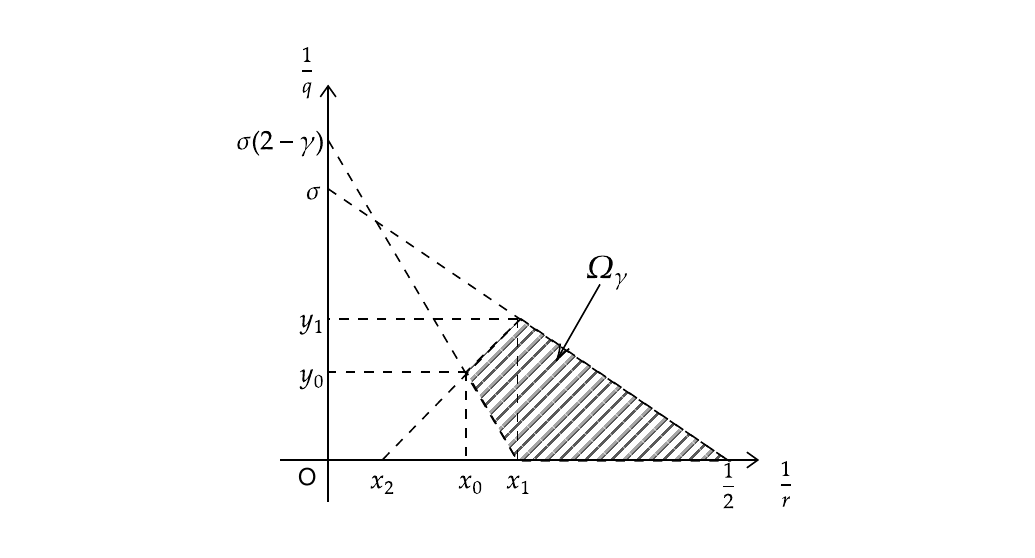}
	\caption{Domain of $\Omega_{\gamma}$}
	\label{fig:omegagamma}
\end{figure}

For $(1/q,1/r) \in \Omega_{\gamma,\sigma}$, denote 
\begin{align*}
	A  =1-\frac{\beta\sigma}{2}+\frac{d-\gamma}{m}.
\end{align*}
Notice that $A$ is positive when $(1/q,1/r) \in \Omega_{\gamma,\sigma}$.

\begin{thm}[GWP]\label{thm-GWP}
	Let 
	\begin{align*}
		&2(1-\eta)<\beta<\frac{2}{\sigma}+2-2\eta, \\
 		&\tilde{\alpha}(q,r)<\frac{1+\sigma(\frac{3}{2}-\eta)}{\beta+1},\\
 		&\frac{1}{r}\in I_{\eta,\sigma},\  (1/r,1/q) \in \Omega_{\gamma,\sigma},\ 	 s\leq0,
	\end{align*}
	and \begin{align*}
		0\leq \theta  <\theta_{\max}:=\Cas{1, & if $\beta(\tilde{\alpha}+1)\leq 1+\dfrac{d-\gamma}{m}$;\\
		\dfrac{A}{\beta(\tilde{\alpha}+1-\dfrac{\sigma}{2})}, &if $\beta(\tilde{\alpha}+1)> 1+\dfrac{d-\gamma}{m}$.}
	\end{align*}
	Then for any $u_{0} \in \brk{L^{2}, X_{q,r}^{s,m}}_{\theta,\infty}$, \eqref{eq-fnlsh} has a unique global solution $u$ written as a sum $u=v+w$, which lie in the spaces
	\begin{align*}
		v\in L_{loc}^{\gamma_{m}(r)}(\real, L_{x}^{r}) \cap L_{loc}^{\infty}  (\real, L_{x}^{2}), \quad w\in L_{loc}^{q}(\real, L_{x}^{r}).
	\end{align*}
\end{thm}

\begin{rem}
	Compared with the methods used in \cite{Bhimani2024Fractional}, Theorem \ref{thm-GWP} provides a conclusion on global well-posedness for a broader range of $p$ and in a more general initial value spaces. 
	
	On the one hand, Bhimani et al. \cite{Bhimani2024Fractional} used a special case of inhomogeneous Strichartz estimates, such that they need $$ \frac{1}{q}=\sigma (\frac{1}{2}-\frac{1}{r}). $$ 
	As shown in the Corollary \ref{coro-local-smooth}, we can also obtain similar global well-posedness for more general $q$ and $r$ by Theorem \ref{thm-GWP}.

	On the other hand, when $m<2$, we could obtain $M_{r,r^{\prime}}\hookrightarrow X_{q,r}^{1/r-1/2,m}$ similarly. The main embedding used in \cite{Bhimani2024Fractional} is $M_{r,r^{\prime}}\hookrightarrow X_{\infty,r}^{1/r-1/2,m}$ for 
	\begin{align}
		\frac{1}{r} & =x_{1}=\frac{2d-\gamma}{4d}.
	\end{align}
	By this embedding and the same method of the proof of Corollary \ref{cor-Mpp'}, we could obtain the same condition of $p_{max}$ as in Theorem 1.2 of \cite{Bhimani2024Fractional}. However, as shown in the proof of Corollary \ref{cor-Mpp'}, if we use the embedding $M_{r,r'} \hookrightarrow X_{q,r}^{1/r-1/2,m}$ when 
	\begin{align}
		\frac{1}{r} & =x_{0}^{+}=\frac{3d-2\gamma}{6d}+,\\
		\frac{1}{q} &=\frac{1}{q_{0}}=\frac{\gamma}{3m},
	\end{align}
	we could obtain a better upper bound of $p_{max}$, which is $p_{m}$ in Corollary \ref{cor-Mpp'}.
	
\end{rem}

\begin{rem}
	We give some comments about the key ingredients of the proof of Theorem \ref{thm-GWP}. For initial data $u_{0} \in  \left(L^{2},  X_{q,r}^{s,m}\right)_{\theta,\infty}$, by Definition \ref{def-real-interpolation}, we know that for any large parameter $N$, we could split $u_{0}$ as a sum of a good function $v_{0}\in L^{2}$  and a bad function $w_{0}\in X_{q,r}^{s,m}$. Then we could obtain a global solution of \eqref{eq-fnlsh} with initial data $v_{0}$ (since the equation is $L^{2}$ subcritical i.e. $\beta<2(1-\dfrac{\gamma}{d})+\dfrac{2m}{d}$) and a global solution for the linear evolution with initial data $w_{0}$. The nonlinear interaction is shown to exist for a small time $T_{0}$ by the contraction mapping argument (which needs the restriction that  $(1/r,1/q)\in \Omega_{\gamma,\sigma}$). This step can be repeated sufficiently many times to obtain $T_{0}, \cdots, T_{N^{\lambda}}$. We choose $T_{k},k=1,\cdots,N^{\lambda}$ and $\lambda$ in a way that the summation $\sum_{k=1}^{N^{\lambda}} T_{k}$ diverges to infinity as $N \rightarrow \infty$ to obtain the global existence of the solution.
\end{rem}

\begin{rem}
	If we consider the fractional nonlinear Schr\"odinger equation
	\begin{align}
		\begin{cases}\label{eq-fnls} \tag{fNLS}
			i \partial_{t} u +(-\Delta)^{m/2} u = \pm \abs{u}^{\beta} u,\\
			u(0,x) = u_{0}(x),
		\end{cases}
	\end{align}
	Theorem \ref{thm-GWP} also holds at the case of $\gamma=d$ by the same argument. To avoid repetition, we omit the proof.
\end{rem}


By the embedding between modulation spaces and $X_{q,r}^{s,m}$, we obtain the global well-posedness of \eqref{eq-fnlsh} in $M_{p,p^{\prime}}^{s}$.

\begin{coro}\label{cor-Mpp'}
	Assume that $\sigma=d/m \leq 1$. Denote 
	\begin{align*}
		\frac{1}{p_{m}}=\Cas{\dfrac{1}{\beta+1}\brk{\dfrac{3}{2}-\dfrac{\gamma}{d}}, & if $1+\dfrac{d-\gamma}{m}\geq \beta(1+\dfrac{\sigma}{2})$; \\ 
		\dfrac{\beta^{2}d(d+2m)+4(d-\gamma)(m-d\beta)+4(d-\gamma)^{2}}{4md\beta(\beta+1)}, & if $1+\dfrac{d-\gamma}{m}< \beta(1+\dfrac{\sigma}{2})$.}
	\end{align*}
	Let $2\leq p<p_{m},s\geq(m-2)(1/2-1/p)$. Then \eqref{eq-fnlsh} is global well-posed from $M_{p,p^{\prime}}^{s}$ to $C(\real,M_{p,p'}+L^{2})$.
\end{coro}

\begin{rem}
	Due to the estimates of fractional nonlinear Schr\"odinger semi-group on modulation spaces (Lemma \ref{lem-mpqs-sm(t)}), our solution must have a loss of regularity $s\geq(m-2)(1/2-1/p)$ when $p \neq 2$. Therefore, the regularity condition in Corollary \ref{cor-Mpp'} is best.
\end{rem}

\begin{rem}
	This assumption $\sigma\leq 1$ is merely for the convenience of stating the conclusion, not essential. Since that when $\sigma>1$, the domain of $\Omega_{\gamma,\sigma}$ need for some classification discussion, the result is relatively complex, but similar conclusions can still be reached using the same method. 
	
	In fact, when $\sigma>1$, denote $$\eta=\frac{\gamma}{d}, \  \tilde{\beta}=\beta-2(1-\eta), \ \beta_{0}=\frac{3-2\eta}{\sigma-1},$$ we could obtain that 
	\begin{align*}
		\Omega_{\gamma,\sigma}= \Cas{\Omega_{\gamma}, & $(2\eta-1)\sigma \leq 2,0<\tilde{\beta}\leq\frac{1}{\sigma}$;\\
		\Omega_{\gamma}\cap \set{x<x_{5}}, & $(2\eta-1)\sigma \leq 2,\frac{1}{\sigma}<\tilde{\beta}\leq\frac{2}{\sigma}$;\\
		\Omega_{\gamma}, & $(2\eta-1)\sigma > 2, 0<\tilde{\beta}\leq \frac{1}{\sigma};$\\
		 \Omega_{\gamma}\cap \set{x<x_{5}}, & $(2\eta-1)\sigma > 2,\frac{1}{\sigma}<\tilde{\beta}\leq  \beta_{0}$;\\
		 \Omega_{\gamma}\cap\set{x_{4}<x<x_{5}}, & $(2\eta-1)\sigma > 2,\beta_{0}<\tilde{\beta}<\frac{2}{\sigma} $.}
	\end{align*}
\end{rem}

%

In particular, when $m=2$, \eqref{eq-fnlsh}  is equivalent to (NLS-h). Combining the local smoothing estimates of Schr\"odinger semi-group on modulation spaces  and Theorem \ref{thm-GWP}, we could obtain global well-posedness of (NLS-h) in $M_{p,2}^{s}$.

\begin{coro}\label{coro-local-smooth}
	  Denote $\sigma=\dfrac{d}{2}, \tilde{\beta}=\beta-2(1-\dfrac{\gamma}{d})$. Let $ \tilde{\beta}>\dfrac{1}{\sigma}\brk{2-\dfrac{\gamma}{d}}$. Then (NLS-h) is global well-posed in $M_{p,2}^{s}$, for any $s>0, 2\leq p <p_{\max}$, where 
	\begin{align*}
		\frac{1}{p_{\max}} =\Cas{\dfrac{\sigma}{2(1+\sigma)}, & $\gamma>\dfrac{d}{2}, \tilde{\beta}>\dfrac{1}{\sigma}(3-\dfrac{2\gamma}{d}), \beta\leq \dfrac{2+d-\gamma}{\sigma+2}$;\\
		\dfrac{4(\beta-1)+3\beta d-2(d-\gamma)}{4\beta(2+d)}, & $\gamma>\dfrac{d}{2}, \tilde{\beta}>\dfrac{1}{\sigma}(3-\dfrac{2\gamma}{d}), \beta> \dfrac{2+d-\gamma}{\sigma+2}$;\\
		 \dfrac{6\beta(d-\gamma)+\brk{2(d-\gamma)-d\beta+2}\brk{d\gamma-2(\beta-1)}}{\beta\brk{2(2+d)(d-\gamma)+(4-d)(d\beta-2)}},& $\gamma>\dfrac{d}{2},\tilde{\beta}\leq\dfrac{1}{\sigma}(3-\dfrac{2\gamma}{d})$ or $\gamma\leq \dfrac{d}{2}$.
		}
	\end{align*}
\end{coro}

\begin{rem}
	Compared with the Theorem 1.6 of \cite{Schippa2021smoothing}, where Schippa obtained the global well-posedness of (1d-cubic-NLS) using the energy method and local smoothing estimates, we only need $s>0$ in Corollary \ref{coro-local-smooth}.
\end{rem}

\begin{rem}
	For the convenience of calculation, the above corollary only considers the case of $m=2$, but similar conclusions can also be obtained for the case of $m>2$.
\end{rem}

\section{Preliminaries} \label{sec-prelimaries}

\subsection{Basic notation}

In this paper, we write $\norm{f}_{p}$ or $\norm{f}_{L^{p}}$ for the usual $L^{p}$ norm in Lebesgue spaces $L^{p}(\real^{d})$. In addition, when $d=1$, for a given interval $I\subseteq \real$, we use the notation $\norm{f}_{L_{I}^{p}}$ to denote the $L^{p}$ norm of $f$ over $I$.  We use the notation $I \lesssim J$ if there is an independent constant C such that $I \leq C J$. Also, we denote $I \approx J$ if $I \lesssim J$ and $J\lesssim I$.

We write $\sch(\real^{d})$ to denote the Schwartz space of all complex-valued rapidly decreasing infinitely  differentiable function on $\real^{d}$, and $\sch^{\prime}(\real^{d})$ to denote the dual space of $\sch(\real^{d})$, called the space of tempered distributions. For simplification, we omit $\real^{d}$ without causing ambiguity. For $1\leq p\leq \infty$, we denote the dual index $p^{\prime}$ by $1/p^{\prime}+1/p=1$, and the $L^{p}$ norm can be defined as follows.
$$
\norm{f}_{p} =\brk{\int_{\real^{d}} \abs{f(x)}^{p} dx}^{1/p}, \ \mbox{when } p<\infty;
$$
and $\norm{f}_{\infty}=\esssup _{x\in \real^{d}} \abs{f(x)}$. The (inverse) Fourier transform $(\FF) \F $ can be defined as follows: 
\begin{align*}
	&\F f(\xi) = \hat{f}(\xi) = \int_{\real^{d}} f(x) e^{- ix\xi} dx;\\
	&\FF f(x) =  (2\pi)^{-d}\int_{\real^{d}}f(\xi) e^{ix\xi} d\xi.
\end{align*}

\subsection{Modulation spaces} \label{subsec-wpqmpq}
Let $s\in \real,1\leq p,q\leq \infty.$ Then modulation spaces $\mpqs$ can be defined as follows.
\begin{align*}
	\mpq^{s} = \sett{ f\in \sch'(\real^{d}), \norm{f}_{\mpq^{s}} < \infty}.
\end{align*}
where the norm $\norm{f}_{\mpq^{s}}$ is defined as 
\begin{align*}
	\norm{f}_{\mpq^{s}} =\norm{ \norm{\inner{\xi}^{s} V_{g}f(x,\xi)}_{L_{x}^{p}}}_{L_{\xi}^{q}}.
\end{align*}
Here $\inner{\xi} = (1+\abs{\xi}^{2})^{1/2}$, $V_{g}f(x,\xi)$ is the short time Fourier transform of the function $f$ with the window function $g\in \sch(\real^{d})\setminus \sett{0}$, that is 
\begin{align*}
	V_{g}f(x,\xi) = \int_{\real} f(t) \overline{g(t-x)} e^{-2\pi it\xi} dt.
\end{align*}

We usual denote $\mpq^{0}$ by $\mpq$. It can be shown that different choices of the window function $g$ lead to equivalent norms (see \cite{Benyi[2020]copyright2020Modulation}). These spaces were first introduced by Feichtinger in \cite{Feichtinger2003Modulation}. The above definition is also referred to as the definition of the continuous version. As for the discrete version of the definition, one can refer to Chapter 5 of \cite{Benyi[2020]copyright2020Modulation}.

Modulation spaces have been studied a lot. Kobayashi and Sugimoto in \cite{Kobayashi2011inclusion} (Guo et al. in \cite{Guo2017Characterization}) gave the sharp inclusion relation between Sobolev and modulation spaces. Here we need an embedding lemma which is a special case of their results.

\begin{lemma}\label{lem-embed-mpq-wpq-Lp}
	Let $2\leq p\leq \infty$. Then we have 
	$$
	M_{p,p'} \hookrightarrow L^{p}.
	$$
\end{lemma} 

B\'enyi, Miyachi, and their cooperators \cite{Benyi2007Unimodular,Miyachi2009Estimates} studied the unimodular Fourier multipliers for modulation spaces. We also refer to \cite{Chen2012Asymptotic,Deng2013Estimate,Zhao2016Sharp}. 

\begin{lemma}[Theorem 4.10; \cite{Zhao2016Sharp}] \label{lem-mpqs-sm(t)}
	Let $m\geq2, 1\leq p,q\leq\infty, s=(m-2)d\abs{1/2-1/p}, t>0$. Then we have
	$$
	\norm{S_{m}(t)u}_{\mpq} \lesssim \inner{t}^{d(1/2-1/p)} \norm{u}_{\mpqs}.
	$$
\end{lemma}

The complex interpolation theorems of modulation spaces are known as follows.
\begin{lemma}[Theorem 2.2; \cite{Han2014$$}] \label{lem-mpq-complex-interpolation}
	Suppose $0<\theta<1$ and 
	$$
	s=(1-\theta)s_{0}+\theta s_{1}, \ \frac{1}{p} = \frac{1-\theta}{p_{0}}+\frac{\theta}{p_{1}},\  \frac{1}{q}=\frac{1-\theta}{q_{0}}+\frac{\theta}{q_{1}}.
	$$
	Then we have 
	\begin{align*}
		&\brk{M_{p_{0},q_{0}}^{s_{0}}, M_{p_{1},q_{1}}^{s_{1}}}_{[\theta]}=\mpqs.
	\end{align*}
\end{lemma}

The scaling properties of modulation spaces have been well studied. 
\begin{lemma}[Scaling of $\mpq$; Theorem 1.1; \cite{Sugimoto2007dilation}] \label{lemma-scaling-mpq}
	For $\lambda>1$, denote $u_{\lambda}(x)=u(\lambda x)$. Then 
	\begin{align*}
		\norm{u_{\lambda}}_{\mpq} & \lesssim \lambda^{d\tau(p,q)} \norm{u}_{\mpq},
	\end{align*}
	where 
	\begin{align*}
		\tau(p,q) & =(-\frac{1}{p})\vee (\frac{1}{q}-\frac{2}{p}) \vee (\frac{1}{q}-1).
	\end{align*}
\end{lemma}
%

\subsection{Some lemmas}

We need to use the (in)homogeneous Strichartz estimate to control the (non)linear terms. Here, we recount these estimates with two lemmas.

\begin{lemma}[Strichartz estimate; Theorem 1.2; \cite{Keel1998Endpoint}]\label{lem-strichartz}
	We say that $(q,r)$ is a $\sigma$-pair, if 
	\begin{align*}
		\Cas{ \dfrac{1}{q}=\sigma \brk{\dfrac{1}{2}-\dfrac{1}{r}},\\
			0<\dfrac{1}{q}\leq \dfrac{1}{2}, \ 0<\dfrac{1}{r}\leq \dfrac{1}{2},\\
			(q,r,\sigma)\neq (2,\infty,1).
		}
	\end{align*}
	For any $\sigma$-pair $(q,r)$, we have 
	\begin{align*}
		\norm{S_{m}(t)u_{0}}_{L_{t}^{q}L_{x}^{r}} & \lesssim \norm{u_{0}}_{2}.
	\end{align*}
\end{lemma}

Denote \begin{align*}
	\Omega^{1} & =\sett{ (x,y) \in [0,1]^2: 0\leq x\leq 1/2, 0<y\leq1, y<2\sigma(1/2-x)} \cup \sett{(1/2,0)}.
\end{align*}
We say that $(q,r)$ is $\sigma$-acceptable, if $(1/r,1/q)\in \Omega^{1}$.

\begin{lemma}[Theorem 1.4; \cite{Foschi2005Inhomogeneous}] \label{lem-inhomogeneous-strichartz}
	Let $(q,r)$, and $(\mu',p')$ be $\sigma$-acceptable with 
	\begin{align*}
		\frac{1}{q} +\frac{\sigma}{r} +1 = \frac{1}{\mu} +\frac{\sigma}{p},
	\end{align*}
	and satisfy one of the following sets of conditions:
	\begin{enumerate}
		\item if $\sigma<1$, there are no further conditions;
		\item if $\sigma=1$, we also require that $r,p'<\infty$;
		\item if $\sigma>1$, we distinguish two cases,
		\begin{itemize}
			\item non-sharp case: 
			\begin{align*}
				&\frac{1}{q}+\frac{1}{\mu'}<1,\\
				&\frac{\sigma-1}{r}\leq \frac{\sigma}{p'}, \  \frac{\sigma-1}{p'}\leq \frac{\sigma}{r};
			\end{align*}
			\item sharp case:
			\begin{align*}
				&\frac{1}{q}+\frac{1}{\mu'}=1,  \\
				&\frac{\sigma-1}{r}<\frac{\sigma}{p'}, \ \  \frac{\sigma-1}{p'}< \frac{\sigma}{r},\\
				&\frac{1}{r} \leq\frac{1}{q}, \ \ \quad \quad \ \frac{1}{p'}\leq \frac{1}{\mu'}.
			\end{align*}
		\end{itemize}
	\end{enumerate}
	Then we have
	\begin{align*}
		\norm{\mathscr{A} f}_{L_{t}^{q}L_{x}^{r}} \lesssim \norm{f}_{L_{t}^{\mu}L_{x}^{p}}.
	\end{align*}

\end{lemma}

\begin{rem}
	The  dispersive inequality (1.2) in \cite{Foschi2005Inhomogeneous} follows from Proposition 3.4 in \cite{Wang2011Harmonic} with $\sigma=d/m$.
\end{rem}

\begin{lemma}[Local smoothing estimate; Theorem 4-B; \cite{Lu2023Local}] \label{lemma-local-smoothing}
	Assume that $$2\leq p\leq p_{0}:=2+\dfrac{4}{d},\  s>\dfrac{d(m-2)}{2}(\dfrac{1}{2}-\dfrac{1}{p}),$$
	then we have 
	\begin{align*}
		\norm{S_{m}(t)u}_{L_{t\in I_{1}}^{p} L_{x}^{p}} & \lesssim \norm{u}_{M_{p,2}^{s}}.
	\end{align*}
\end{lemma}

\section{Some propositions}

We first show that the fractional Schr\"odinger semi-group is bounded on $X_{q,r}^{s,m}$.

\begin{prop}\label{prop-schrodinger-boundedness-Xqrsm}
	For any $t_{0}>0$, we have 
	\begin{align*}
		\norm{S_{m}(t_{0})u_{0}}_{X_{q,r}^{s,m}} & \leq \inner{t_{0}}^{\abs{s}} \norm{u_{0}}_{X_{q,r}^{s,m}}.
	\end{align*}
\end{prop}

\begin{proof}
	By Definition \ref{def-Xqr-sigma}, we have 
	\begin{align*}
		\inner{T}^{s}\norm{ S_{m}(t+T) S_{m}(t_{0})u_{0}}_{L_{I_{1}}^{q}L_{x}^{r}} &=\inner{T}^{s} \norm{S_{m}(t+T+t_{0})u_{0}}_{L_{I_{1}}^{q}L_{x}^{r}} \\
		&=\inner{T}^{s} \inner{T+t_{0}}^{-s} \inner{T+t_{0}}^{s} \norm{S_{m}(t+T+t_{0})u_{0}}_{L_{I_{1}}^{q}L_{x}^{r}}\\
		&\leq \inner{t_{0}}^{\abs{s}} \norm{u_{0}}_{X_{q,r}^{s,m}},
	\end{align*}
	where we use the inequality $\inner{a+b}^{s}\leq \inner{a}^{s} \inner{b}^{\abs{s}}$ at the last inequality. 
	Taking the supremum over $T>0$, we could obatin the estimate as desired.
\end{proof}

Here, we present a proposition to explain why we need $(1/r,1/q)\in \Omega_{\gamma,\sigma}$ in Theorem \ref{thm-GWP}.

\begin{prop}\label{prop-Omega-gamma-sigma}
	If $(1/r,1/q)\in \Omega_{\gamma,\sigma}$, we denote 
	\begin{align*}
		\frac{1}{p} &=\frac{\beta+1}{r}+\frac{\gamma}{d}-1,  \\
		\frac{1}{\mu} & =1+\frac{1}{q}+\sigma \brk{1-\frac{\gamma}{d}-\frac{\beta}{r}}.
	\end{align*}
	Then $(q,r)$ and $(\mu',p')$ satisfy the conditions in Lemma \ref{lem-inhomogeneous-strichartz}, which means that we have
	\begin{align*}
		\norm{\mathscr{A} f}_{L_{t}^{q}L_{x}^{r}} \lesssim \norm{f}_{L_{t}^{\mu}L_{x}^{p}}.
	\end{align*}
\end{prop}

\begin{proof}
	We first consider the case of $\sigma\leq 1$.
	
	If we want that $(q,r)$ and $(\mu',p')$ are $\sigma$-acceptable with $r,p'<\infty$, we only need 
	\begin{align*}
		\Cas{0<\dfrac{1}{q}<1,  0<\dfrac{1}{r}, \dfrac{1}{q}<2\sigma\brk{\dfrac{1}{2}-\dfrac{1}{r}}, \\
			0<\dfrac{1}{\mu'}<1,  0<\dfrac{1}{p'}, \dfrac{1}{\mu'}<2\sigma \brk{\dfrac{1}{2}-\dfrac{1}{p'}}.
		}
	\end{align*}
	Substituting the equations of $(1/p,1/\mu)$, we have 
		\begin{align*}
		\Cas{0<\dfrac{1}{q}<1,  0<\dfrac{1}{r}, \dfrac{1}{q}<2\sigma\brk{\dfrac{1}{2}-\dfrac{1}{r}}, \\
		\dfrac{1}{q}<\sigma\brk{\dfrac{\beta}{r}-1+\dfrac{\gamma}{d}},\\
		\dfrac{1}{q}>\sigma \brk{2-\dfrac{\gamma}{d}-\dfrac{\beta+2}{r}},\\
		\dfrac{1}{r}<\dfrac{2-\dfrac{\gamma}{d}}{\beta+1},
		}
	\end{align*}
	which means $(1/r,1/q)\in \Omega_{\gamma,\sigma}$.
	
	Later we consider the case of $\sigma> 1$. 
	
	\begin{itemize}
		\item non-sharp case: we only need that 
		\begin{align*}
			\Cas{\dfrac{1}{q}+\dfrac{1}{\mu'}<1,\\
			\dfrac{\sigma-1}{r}<\dfrac{\sigma}{p'},\\
			\dfrac{\sigma-1}{p'}<\dfrac{\sigma}{r}.}
		\end{align*}
		Substituting the equations of $(1/p,1/\mu)$, we have 
		\begin{align*}
			\Cas{\dfrac{1}{r}<\dfrac{m+d-\gamma}{d\beta},\\
				\dfrac{1}{r}<\dfrac{\sigma(2-\dfrac{\gamma}{d})}{(\beta+2)\sigma-1},\\
				\dfrac{1}{r}<\dfrac{(\sigma-1)(2-\dfrac{\gamma}{d})}{(\beta+2)\sigma-1},}
		\end{align*}
		which is equivalent to $x_{4}<1/r<x_{5}$ since that $\dfrac{m+d-\gamma}{d\beta}>\dfrac{1}{2}$ under the condition of \eqref{eq-beta-condition}.
		\item sharp case: by $1/q+1/\mu'=1$, we have $\dfrac{1}{r}=\dfrac{m+d-\gamma}{d\beta}>\dfrac{1}{2}$, which is impossible.
	\end{itemize}
\end{proof}

Denote $G(v,w)= \brk{K*\abs{v+w}^{\beta}}(v+w)-\brk{K*\abs{v}^\beta }v$.

\begin{prop}\label{prop-LqLr-estimate}
	Let $(1/r,1/q) \in \Omega_{\gamma,\sigma}$. Recall that \begin{align*}
		A=1-\dfrac{\beta\sigma}{2}+\dfrac{d-\gamma}{m}, \ \tilde{\alpha}=\dfrac{1}{q}+\dfrac{\sigma}{r}, \ 
		\dfrac{1}{\gamma_{m}(r)}=\sigma(\dfrac{1}{2}-\dfrac{1}{r}).
	\end{align*} Denote 
	\begin{align*}
		D  =A+\frac{\beta\sigma}{2}-\beta\tilde{\alpha},\  
		E  =A+\frac{\sigma}{2}-\tilde{\alpha}, \ 
		F =A+\frac{(\beta+1)\sigma}{2}-(\beta+1)\tilde{\alpha}.
	\end{align*} Then we have
	\begin{align*}
		\norm{\mathscr{A} G(v,w)}_{L_{I_{T}}^{q} L_{x}^{r}} \lesssim \brk{T^{A} \norm{v}_{L_{I_{T}}^{\gamma_{m}(r)} L_{x}^{r}}^{\beta} + T^{D} \norm{w}_{L_{I_{T}}^{q} L_{x}^{r}}^{\beta}} \norm{w}_{L_{I_{T}}^{q} L_{x}^{r}};\\
		\norm{\mathscr{A} G(v,w)}_{L_{I_{T}}^{\infty} L_{x}^{2}} \lesssim \brk{T^{E} \norm{v}_{L_{I_{T}}^{\gamma_{m}(r)} L_{x}^{r}}^{\beta} + T^{F} \norm{w}_{L_{I_{T}}^{q} L_{x}^{r}}^{\beta}} \norm{w}_{L_{I_{T}}^{q} L_{x}^{r}}.
	\end{align*}
\end{prop}

\begin{proof}
	By Lemma \ref{lem-inhomogeneous-strichartz} and Proposition \ref{prop-Omega-gamma-sigma}, we have 
	\begin{align}\label{eq-AG}
		\norm{\mathscr{A} G(v,w)}_{L_{I_{T}}^{q} L_{x}^{r}} \lesssim \norm{G(v,w)}_{L_{I_{T}}^{\mu} L_{x}^{p}},
	\end{align}
	where 
	\begin{align*}
		\frac{1}{p} =\frac{\beta+1}{r}+\frac{\gamma}{d}-1,  \quad 
		\frac{1}{\mu}  =1+\frac{1}{q}+\sigma \brk{1-\frac{\gamma}{d}-\frac{\beta}{r}}.
	\end{align*}
	Notice that 
	\begin{align*}
		\abs{G(v,w)}  & \lesssim \brk{K*(\abs{v}^{\beta-1}\abs{w})}\abs{v}+\brk{K*\abs{w}^{\beta}}\abs{w} +\brk{K*\abs{v}^{\beta}} \abs{w}.
	\end{align*}
	By H\"older's inequality and Hardy-Littlewood-Sobolev inequality, we have
	\begin{align*}
		\eqref{eq-AG}&\lesssim \norm{K*(\abs{v}^{\beta-1}\abs{w})}_{L_{I_{T}}^{\mu_{1}} L_{x}^{p_{1}}} \norm{v}_{L_{I_{T}}^{\gamma_{m}(r)} L_{x}^{r}}+ \norm{K*\abs{w}^{\beta}}_{L_{I_{T}}^{\mu_{2}} L_{x}^{p_{1}}} \norm{v}_{L_{I_{T}}^{q} L_{x}^{r}}+\norm{K*\abs{v}^{\beta}}_{L_{I_{T}}^{\mu_{2}} L_{x}^{p_{1}}} \norm{v}_{L_{I_{T}}^{q} L_{x}^{r}}\\
		&\lesssim \norm{\abs{v}^{\beta-1} w}_{L_{I_{T}}^{\mu_{1}} L_{x}^{p_{2}}}\norm{v}_{L_{I_{T}}^{\gamma_{m}(r)} L_{x}^{r}}+\norm{w^{\beta}}_{L_{I_{T}}^{\mu_{2}} L_{x}^{p_{2}}} \norm{w}_{L_{I_{T}}^{q} L_{x}^{r}} +\norm{v^{\beta}}_{L_{I_{T}}^{\mu_{2}} L_{x}^{p_{2}}} \norm{w}_{L_{I_{T}}^{q} L_{x}^{r}} \\
		&\lesssim \brk{T^{A} \norm{v}_{L_{I_{T}}^{\gamma_{m}(r)} L_{x}^{r}}^{\beta} + T^{D} \norm{w}_{L_{I_{T}}^{q} L_{x}^{r}}^{\beta}} \norm{w}_{L_{I_{T}}^{q} L_{x}^{r}}.
	\end{align*}
	where 
	\begin{align*}
		\begin{cases}
			\dfrac{1}{\mu}=\dfrac{1}{\mu_{1}}+\dfrac{1}{\gamma_{m}(r)}=\dfrac{1}{\mu_{2}}+\dfrac{1}{q},\\
			\dfrac{1}{p}=\dfrac{1}{p_{1}}+\dfrac{1}{r},\\
			\dfrac{1}{p_{1}}+1=\dfrac{1}{p_{2}}+\dfrac{\gamma}{d},\\
			\dfrac{1}{\mu_{1}}=A+\dfrac{\beta-1}{\gamma_{m}(r)}+\dfrac{1}{q},\\
			\dfrac{1}{\mu_{2}}=D+\dfrac{\beta}{q}=A+\dfrac{\beta}{\gamma_{m}(r)}.
		\end{cases}
	\end{align*}
	By the same argument, we could obtain the estimate of $\norm{\mathscr{A} G(v,w)}_{L_{I_{T}}^{\infty} L_{x}^{2}}$. 
	
	To ensure that the above process is meaningful, we need to ensure that $D, E,$ and $F$ are positive. By the method of linear programming, we can obtain the upper bound of $\tilde{\alpha}$ is $\dfrac{d\beta+\gamma}{m(\beta+2)}$ when $(1/r,1/q) \in \Omega_{\gamma,\sigma}$, which can ensure that $D, E,$ and $F$ are positive.
\end{proof}

\begin{lemma}[LWP in $L^2$] \label{lem-LWP-L2}
	Let $v_{0}\in L^{2}$ be the initial data $u_{0}$ in \eqref{eq-fnlsh}. Then there exists a unique solution $v\in L_{I_{T} }^{\gamma_{m}(r)}L_{x}^{r}\cap L_{I_{T} }^{\infty}L_{x}^{2}$ of \eqref{eq-fnlsh}, with \begin{align*}
		\norm{v}_{L_{I_{T} }^{\gamma_{m}(r)}L_{x}^{r}}  \lesssim \norm{v_{0}}_{2},\quad
		\norm{v}_{L_{I_{T} }^{\infty}L_{x}^{2}} = \norm{v_{0}}_{2},
	\end{align*}
	for any $ T^{A} \norm{v_{0}}_{2}^{\beta} \ll 1$.
\end{lemma}

\begin{proof}
	This proof is a routine and we only give a sketch here.
	
	To obtain the solution $v$ satisfying
	\begin{align*}
		v=S_{m}(t)v_{0}\mp i\mathscr{A} \brk{\brk{K*\abs{v}^{\beta}}v},
	\end{align*}
	we define a operator $$\mathscr{L}:  v \longrightarrow S_{m}(t)v_{0}\mp i\mathscr{A} \brk{\brk{K*\abs{v}^{\beta}}v}.$$ 
	
	By the Strichartz estimate of fractional Schr\"odinger semigroup(Lemma \ref{lem-strichartz}), Lemma \ref{lem-inhomogeneous-strichartz}, and H\"older's inequality, we have 
	\begin{align*}
		\norm{\mathscr{L}v}_{L_{I_{T} }^{\gamma_{m}(r)}L_{x}^{r} \cap L_{I_{T} }^{\infty}L_{x}^{2}} \lesssim \norm{v_{0}}_{2}+T^{A} \norm{v}_{L_{I_{T} }^{\gamma_{m}(r)}L_{x}^{r}}^{\beta+1}.
	\end{align*}
	
	For $T^{A} \norm{v_{0}}_{2}^{\beta} \ll 1$, we denote $$D_{0}=\sett{v\in L_{I_{T} }^{\gamma_{m}(r)}L_{x}^{r}: \norm{v}_{L_{I_{T} }^{\gamma(r)}L_{x}^{r}} \lesssim \norm{v_{0}}_{2}} .$$
	One   can prove that the operator 
	$
	\mathscr{L}: D_{0} \longrightarrow D_{0}
	$
	is a contraction.  The local well-posedness follows from the constraction mapping argument.
\end{proof}

\begin{align}
	\begin{cases}\label{eq-fnlsh-v} \tag{fNLS-v}
		i \partial_{t} w +(-\Delta)^{m/2} w =  \pm G(v,w),\\
		w(0,x) = w_{0}.
	\end{cases}
\end{align}

Notice that the sum of the solution $w$ of \eqref{eq-fnlsh-v} with initial data $w_{0}$ and the solution $v$ of \eqref{eq-fnlsh}
with initial data $v_{0}$ is just the solution of \eqref{eq-fnlsh} with initial data $u_{0}=w_{0}+v_{0}$. Therefore, we need to consider the well-posedness of \eqref{eq-fnlsh-v}.

\begin{lemma}[LWP of \eqref{eq-fnlsh-v}] \label{lem-fnlsv}
	Let $0<T_{0}<1, (1/q,1/r) \in \Omega_{\gamma,\sigma},  v\in L_{I_{T_{0}}}^{\gamma_{m}(r)} L_{x}^{r}, S(t)w_{0} \in L_{I_{1}}^{q} L_{x}^{r}$. Then there exists a unique solution $w\in L_{I_{T}}^{q} L_{x}^{r}$ of \eqref{eq-fnlsh-v}, with \begin{align*}
		&\norm{w}_{L_{I_{T}}^{q} L_{x}^{r}}  \lesssim \norm{S(t)w_{0}}_{L_{I_{T}}^{q} L_{x}^{r}}; \\
		&\norm{w-S(t)w_{0}}_{L_{I_{T}}^{\infty} L_{x}^{2}}  \lesssim \brk{T^{E} \norm{v}_{L_{I_{T}}^{\gamma_{m}(r)} L_{x}^{r}}^{2} + T^{F} \norm{w}_{L_{I_{T}}^{q} L_{x}^{r}}^{2}} \norm{w}_{L_{I_{T}}^{q} L_{x}^{r}}.
	\end{align*}
	Here we assume that $T\leq T_{0}, T^{A} \norm{v}_{L_{I_{T_{0}}}^{\gamma_{m}(r)} L_{x}^{r}}^{\beta} + T^{D} \norm{S_{m}(t)w_{0} }_{L_{I_{1}}^{q} L_{x}^{r}}^{2}\ll 1$. 
\end{lemma}

\begin{proof}
	Let $0<T\leq T_{0}$ to be determined, we define a operator $$\mathscr{T}:  w \longrightarrow S_{m}(t)w_{0} \mp i \mathscr{A} G(v,w).$$ By Proposition \ref{prop-LqLr-estimate},  there exists $C>1$, such that 
	\begin{align}\label{eq-fnlshv-lwp}
		&\norm{\mathscr{T} w}_{L_{I_{T}}^{q}L_{x}^{r}} \leq \norm{S_{m}(t)w_{0}}_{L_{I_{T}}^{q}L_{x}^{r}} +C \left(T^{A}\norm{v}_{L_{I_{T}}^{\gamma_{m}(r)}L_{x}^{r}}^{\beta}+T^{D} \norm{w}_{L_{I_{T}}^{q}L_{x}^{r}}^{\beta}  \right) \norm{w}_{L_{I_{T}}^{q}L_{x}^{r}}.
	\end{align}
	For any $T>0$ with
	\begin{align}\label{eq-T-condition}
		T^{A} \norm{v}_{L_{I_{T_{0}}}^{\gamma_{m}(r)} L_{x}^{r}}^{\beta} + T^{D} \norm{S(t)w_{0} }_{L_{I_{1}}^{q} L_{x}^{r}}^{\beta}\leq \frac{1}{8C},
	\end{align}
	we denote 
	$$
	D_{1}=\sett{w\in L_{I_{T}}^{q}L_{x}^{r}: \norm{w}_{L_{I_{T}}^{q}L_{x}^{r}} \leq 2  \norm{S_{m}(t)w_{0}}_{L_{I_{T}}^{q}L_{x}^{r}}}.
	$$ Then by \eqref{eq-fnlshv-lwp} and \eqref{eq-T-condition}, we can prove that the operator 
	$$
	\mathscr{T}: D_{1} \longrightarrow D_{1}
	$$ 
	is a contraction. We can first obtain a similar estimate like Proposition 10 in \cite{Chaichenets2017existence}. Then the proof of the contraction is a routine. We omit it.
	
	Then there exists a unique local solution $w$ of \eqref{eq-fnlsh-v} on $[0,T]$ with 
	\begin{align*}
		w=S_{m}(t)w_{0} \mp i \mathscr{A} G(v,w).
	\end{align*}
	By Proposition \ref{prop-LqLr-estimate}, we obtain the estimate of $\norm{w-S(t)w_{0}}_{L_{I_{T}}^{\infty} L_{x}^{2}}$ as desired in the proposition.
	
\end{proof}

\section{Proof of main results}

\begin{proof}[Proof of Theorem \ref{thm-lwp}]
	Denote $A=q \wedge \gamma_{m}(r), I_{T}=[0,T]$, and define the work space:
	\begin{align*}
		&X := L_{t\in I_{T}}^{A}L_{x}^{r}L_{x}^{r} \cap C(I_{T}, X_{q,r}^{s,m}+L_{2}),\\
		&B_{X}(R):= \set{u\in X: \norm{u}_{X}\leqslant R}.
	\end{align*}
	For any $u_{0}\in X_{q,r}^{s,m}+L^{2}$, we could define the map:
	\begin{align}
		\mathscr{T}:  \quad & u \rightarrow S_{m}(t)u_{0}\mp i\mathscr{A} \brk{K*\abs{u}^{\beta}} u.
	\end{align}
	We only need to prove that $\mathscr{T}: B_{X}(R) \rightarrow B_{X}(R)$ is a constraction for certain $0<T<1, 0<R$. Then, by the standard contraction mapping argument, we see that \eqref{eq-fnlsh} has a unique solution $B_{X}(R)$.
	
	We first consider the linear term $S_{m}(t)u_{0}$. By Proposition \ref{prop-schrodinger-boundedness-Xqrsm} and the isometry property of $S_{m}(t)$ on $L^{2}$, we have 
	\begin{align} \label{eq-linear-1}
		\norm{S_{m}(t)u_{0}}_{L_{I_{T}}^{\infty}(X_{q,r}^{s,m}+L^{2})} & \leq C_{0}\norm{u_{0}}_{X_{q,r}^{s,m}+L^{2}}.
	\end{align}
	For any $u_{0}=v_{0}+w_{0}\in X_{q,r}^{s,m}+L^{2}$, by H\"older's inequality of $t\in I_{T}$ and Strichartz estimate (Lemma \ref{lem-strichartz}), we have 
	\begin{align*}
		\norm{S_{m}(t)u_{0}}_{L_{t\in I_{T}}^{A}L_{x}^{r}} & \leq \norm{S_{m}(t)v_{0}}_{L_{t\in I_{T}}^{A}L_{x}^{r}} +\norm{S_{m}(t)w_{0}}_{L_{t\in I_{T}}^{A}L_{x}^{r}} \\
		&\leq T^{\frac{1}{A}-\frac{1}{q}} \norm{S_{m}(t)v_{0}}_{L_{t\in I_{T}}^{q}L_{x}^{r}} +T^{\frac{1}{A}-\frac{1}{\gamma_{m}(r)}}\norm{S_{m}(t)w_{0}}_{L_{t\in I_{T}}^{\gamma_{m}(r)}L_{x}^{r}} \\
		 & \leq T^{\frac{1}{A}-\frac{1}{q}} \norm{v_{0}}_{X_{q,r}^{s,m}}+T^{\frac{1}{A}-\frac{1}{\gamma_{m}(r)}} \norm{w_{0}}_{2},\\
		 &\leq \norm{v_{0}}_{X_{q,r}^{s,m}}+\norm{w_{0}}_{2}.
	\end{align*}
	Therefore, we have 
	\begin{align}\label{eq-linear-2}
		\norm{S_{m}(t)u_{0}}_{L_{t\in I_{T}}^{A}L_{x}^{r}} & \leq \norm{u_{0}}_{X_{q,r}^{s,m}+L^{2}}.
	\end{align}
	
	Then we give the estimates of nonlinear term $\mathscr{A} \brk{K*\abs{u}^{\beta}} u$.
	By the embedding $L^{2}\hookrightarrow X_{q,r}^{s,m}+L^{2}$, we have 
	\begin{align}\label{eq-nonlinear-1}
		\norm{\mathscr{A} \brk{K*\abs{u}^{\beta}} u}_{L_{I_{T}}^{\infty}(X_{q,r}^{s,m}+L^{2})} & \leq \norm{\mathscr{A} \brk{K*\abs{u}^{\beta}} u}_{L_{I_{T}}^{\infty}L^{2}}.
	\end{align}
	By H\"older's inequality of $t\in I_{T}$, we have 
	\begin{align}\label{eq-nonlinear-2}
		\norm{\mathscr{A} \brk{K*\abs{u}^{\beta}} u}_{L_{t\in I_{T}}^{A}L_{x}^{r}} & \leq \norm{\mathscr{A} \brk{K*\abs{u}^{\beta}} u}_{L_{t\in I_{T}}^{\gamma_{m}(r)}L_{x}^{r}}.
	\end{align}
	By the inhomogeneous Strichartz estimate (Lemma \ref{lem-inhomogeneous-strichartz}), H\"older's inequality of $x,t$, and Hardy-Littlewood-Sobolev inequality, we have 
	\begin{align}\label{eq-nonlinear-3}
		\norm{\mathscr{A} \brk{K*\abs{u}^{\beta}} u}_{L_{t\in I_{T}}^{\gamma_{m}(r)}L_{x}^{r} \cap L_{I_{T}}^{\infty}L^{2}} & \leq C_{1} \norm{\brk{K*\abs{u}^{\beta}} u }_{L_{t\in I_{T}}^{\mu}L_{x}^{p}}\notag\\ 
		&\leq C_{1} \norm{K*\abs{u}_{\beta}}_{L_{t\in I_{T}}^{\mu_{1}}L_{x}^{p_{1}}} \norm{u}_{L_{t\in I_{T}}^{A}L_{x}^{r}}\notag\\
		& \leq C_{2} \norm{\abs{u}^{\beta}}_{L_{t\in I_{T}}^{\mu_{1}}L_{x}^{p_{2}}} \norm{u}_{L_{t\in I_{T}}^{A}L_{x}^{r}}\notag\\
		&\leq C_{2} T^{\eps \beta} \norm{u}_{L_{t\in I_{T}}^{A}L_{x}^{r}}^{\beta+1},
	\end{align}
	where 
	\begin{align*}
		\Cas{
		\dfrac{1}{p}=\dfrac{1}{p_{1}} +\dfrac{1}{r},\\
		\dfrac{1}{p_{1}}+1=\dfrac{1}{p_{2}}+\dfrac{\gamma}{d},\\
		\dfrac{1}{\beta p_{2}}=\dfrac{1}{r},\\
		\dfrac{1}{\mu}=\dfrac{1}{\mu_{1}}+\dfrac{1}{A},\\
		\eps=\dfrac{1}{\beta\mu_{1}}-\dfrac{1}{A}>0,\\
		(\gamma_{m}(r),r), (\mu',p') \text{ are $\sigma$-pairs.}
		} 
	\end{align*}
	These equations are equivalent to 
	\begin{align*}
		\Cas{ \dfrac{1}{p}=\dfrac{\beta+1}{r}+\dfrac{\gamma}{d}-1,\\
			\dfrac{1}{\mu}=1-\sigma(\dfrac{1}{p}-\dfrac{1}{2}),\\
			\dfrac{1}{2}\leq \dfrac{1}{p}\leq 1, \dfrac{1}{2}\leq \dfrac{1}{\mu}\leq 1,\\
			\dfrac{1}{\mu}>\dfrac{\beta+1}{A},\\
			\dfrac{\sigma-1}{2\sigma}\leq \dfrac{1}{r}\leq \dfrac{1}{2}.
		} 
	\end{align*}
	After simplification, it becomes: 
	\begin{align*}
		\Cas{a(\eta,\sigma) \leq \dfrac{1}{r} \leq b(\eta,\sigma),\\
			\dfrac{1}{q}+\dfrac{\sigma}{r}<\dfrac{1+\sigma(\frac{3}{2}-\eta)}{\beta+1},
		} 
	\end{align*}
	which is the assumptions in Theorem \ref{thm-lwp}. 
	
	Combining \eqref{eq-linear-1}-\eqref{eq-nonlinear-3}, we have 
	\begin{align*}
		\norm{\mathscr{T} u}_{X} & \leq C_{0} \norm{u_{0}}_{X_{q,r}^{s,m}+L^{2}} +C_{2} T^{\eps\beta} \norm{u}_{X}^{\beta+1}.
	\end{align*}
	Therefore, if we choose $R=2C_{0}\norm{u_{0}}_{X_{q,r}^{s,m}+L^{2}}$ and $C_{2}T^{\eps\beta} R^{\beta} =1/2$, we know that $$\mathscr{T}: \ B_{R}(X) \rightarrow B_{R}(X)$$ is a constraction. 
\end{proof}

\begin{proof}[Proof of Theorem \ref{thm-GWP}]
	For any \begin{align*}
		u_{0} \in \left(L^{2}, X_{q,r}^{s,m}\right)_{\theta,\infty},
	\end{align*}
	by Definition \ref{def-real-interpolation}, we know that for any $N\gg 1$, there exists a decomposition $u_{0}=v_{0}+w_{0} \in L^{2}+X_{q,r}^{s,m}$, such that 
	\begin{align}\label{eq-u0-v0w0}
		\norm{v_{0}}_{2} \leq N^{\alpha},\quad \norm{w_{0}}_{X_{q,r}^{s,m}} \lesssim N^{-1},
	\end{align} 
	where $\alpha=\theta/(1-\theta)(\Longleftrightarrow \theta=\alpha/(1+\alpha))$.
	By Definition \ref{def-Xqr-sigma}, we have 
	\begin{align*}
		\norm{S_{m}(t)w_{0}}_{L_{I_{1}}^{q}L_{x}^{r}} \lesssim \norm{w_{0}}_{X_{q,r}^{s,m}}.
	\end{align*}
	Then by Lemmas \ref{lem-LWP-L2} and \ref{lem-fnlsv}, there exists a solution of \eqref{eq-fnlsh} $u^{(0)}=v^{(0)}+w^{(0)}$ on $[0,T_{0}]$ with initial data $u_{0}=v_{0}+w_{0}$. 
	
	If there exists a decomposition of $u^{(0)}(T_{0})=v_{1}+w_{1}$ with $v_{1}\in L^{2}, S(t)w_{1} \in L_{I_{1}}^{q}L_{x}^{r}$ (we assume the existence of this decomposition  here, and will construct a decomposition later), then there exists a solution $u^{(1)}=v^{(1)}+w^{(1)}$ of \eqref{eq-fnlsh}  on $[0,T_{1}]$ with initial data $u^{(0)}(T_{0})$, which means that \eqref{eq-fnlsh} has a local solution on $[0,T_{0}+T_{1}]$. 
	
	By induction (under our assumptions), we obtain the local solution $u^{(k)}=v^{(k)}+w^{(k)}$ on $[0,T_{k}]$ (we always assume that the initial time of the solution $u^{(k)}$ is $t=0$). If there exists a decomposition 
	\begin{align} \label{eq-decomposition-at-T}
		u^{(k)}(T_{k}) = v_{k+1}+w_{k+1}
	\end{align}
	satisfying the conditions in Lemmas \ref{lem-LWP-L2} and \ref{lem-fnlsv}, we could obtain the solution $$u^{(k+1)}=v^{(k+1)}+w^{(k+1)}$$ on $[0,T_{k+1}]$. Combining all these $u^{(\ell)}, 0\leq \ell \leq k+1$,  we know that \eqref{eq-fnlsh} has a solution on $[0,\sum_{\ell=0}^{k+1} T_{\ell}]$.
	
	The critical aspect here is the establishment of the decomposition \eqref{eq-decomposition-at-T} to control $T_{\ell}$ from below in order to have $\sum_{\ell} T_{\ell} =\infty$. We construct a decomposition
	\begin{align*}
		&v_{\ell+1}= v^{(\ell)}(T_{\ell})+w^{(\ell)}(T_{\ell})-S_{m}(T_{\ell})w_{\ell} =u^{(\ell)}(T_{\ell})- S_{m}(T_{\ell})w_{\ell};\\
		&w_{\ell+1}= u^{(\ell)}(T_{\ell})-v_{\ell+1}=S_{m}(T_{\ell})w_{\ell}.
	\end{align*}
	This decomposition comes from Lemmas \ref{lem-LWP-L2} and \ref{lem-fnlsv}, where we know that $v^{(\ell)}(T_{\ell}), w^{(\ell)}(T_{\ell})-S_{m}(T_{\ell})w_{\ell}\in L^{2}$. So we sum them up to be $v_{\ell+1}$, and put $w_{\ell+1}=u^{(\ell)}(T_{\ell})-v_{\ell+1}$.
	
	By Definition \ref{def-Xqr-sigma}, we have 
	\begin{align}\label{eq-S(t)wk}
		\norm{S_{m}(t)w_{\ell+1}}_{L_{I_{1}}^{q} L_{x}^{r}} & = \norm{S _{m}(t+T_{\ell})w_{\ell}}_{L_{I_{1}}^{q} L_{x}^{r}} \notag\\
		&= \norm{S_{m}(t+T_{\ell}+\cdots + T_{0})w_{0}}_{L_{I_{1}}^{q} L_{x}^{r}}\notag\\
		&\lesssim \inner{T_{\ell}+\cdots + T_{0}}^{-s} \norm{w_{0}}_{X_{q,r}^{\sigma,m}}\notag\\
		&\lesssim \inner{T_{\ell}+\cdots + T_{0}}^{-s} \norm{w_{0}}_{X_{q,r}^{s,m}} \lesssim \inner{T_{\ell}+\cdots + T_{0}}^{-s} N^{-1}.
	\end{align}
	By Proposition \ref{prop-LqLr-estimate} and Lemma \ref{lem-LWP-L2}, there exists certain $C_{0}>1$ ($C_{0}$ may be different in different estimates), such that
	\begin{align*}
		\norm{v_{\ell+1}}_{2} &= \norm{v^{(\ell)}(T_{\ell})+w^{(\ell)}(T_{\ell})-S_{m}(T_{\ell})w_{\ell}}_{2} \\
		&\leq \norm{v^{(\ell)}(T_{\ell})}_{2} + \norm{w^{(\ell)}(T_{\ell})-S_{m}(T_{\ell})w_{\ell}}_{2}\\
		& \leq \norm{v_{\ell}}_{2} +C_{0} \left( T_{\ell}^{E} \norm{v_{\ell}}_{2}^{\beta} + T_{\ell}^{F} \norm{S_{m}(t)w_{\ell}}_{L_{I_{T_{\ell}}}^{q}L_{x}^{r}}^{\beta}\right) \norm{S_{m}(t)w_{\ell}}_{L_{I_{T_{\ell}}}^{q}L_{x}^{r}}\\
		&\leq \norm{v_{\ell}}_{2} +C_{0} \left( T_{\ell}^{E} \norm{v_{\ell}}_{2}^{\beta} + T_{\ell}^{F} \left(\inner{T_{\ell}+\cdots + T_{0}}^{-s} N^{-1}\right)^{\beta}\right) \inner{T_{\ell}+\cdots + T_{0}}^{-s} N^{-1}.
	\end{align*}
	Here we use the $L^{2}$-conservation law of $v^{\ell}$.
	
	Taking summation from $\ell=0$ to $k$, we have 
	\begin{align}\label{eq-vk-v0}
		\norm{v_{k+1}}_{2} &\leq \norm{v_{0}}_{2} \notag \\
		&+ C_{0} \sum_{\ell=0}^{k} \left( T_{\ell}^{E} \norm{v_{\ell}}_{2}^{\beta} + T_{\ell}^{F} \left(\inner{T_{\ell}+\cdots + T_{0}}^{-s} N^{-1}\right)^{\beta}\right) \inner{T_{\ell}+\cdots + T_{0}}^{-s} N^{-1}.
	\end{align}
	For $0\leq k \leq K_{N}\approx N^{\lambda}$($\lambda$ will be determined later),  choose $T_{k} =c_{0} N^{-\alpha\beta/A}$, where $c_{0}$ is small enough and will be determined later. By Lemmas \ref{lem-LWP-L2} and \ref{lem-fnlsv}, we need that 
	\begin{align}\label{eq-Tk-condition}
		T_{k}^{A} \norm{v_{k}}_{2}^{\beta} +T_{k}^{A} \norm{v^{(k)}}_{L_{I_{T}}^{\gamma_{m}(r)} L_{x}^{r}}^{\beta} + T_{k}^{D} \norm{S_{m}(t)w_{k} }_{L_{I_{1}}^{q} L_{x}^{r}}^{\beta}\ll 1,
	\end{align} 
	holds for any $k=0,\cdots,K_{N}$.
	
	By induction, we can obtain that $\norm{v_{k}}_{2}\leq 2N^{\alpha}$  for any $k=0,\cdots,K_{N}$.
	
	At first, by \eqref{eq-u0-v0w0}, \eqref{eq-S(t)wk} and the assumption that $s\leq 0$, we have
	$$\norm{v_{0}}_{2} \leq N^{\alpha} \leq 2N^{\alpha}, \ \norm{S(t)w_{k}}_{L_{I_{1}}^{q} L_{x}^{r}} \lesssim \inner{T_{\ell}+\cdots + T_{0}}^{-s} N^{-1}\lesssim (kN^{-\alpha\beta/A})^{-s}N^{-1}.$$ 
	If we have $\norm{v_{\ell}}_{2}\leq 2 N^{\alpha}, \forall \ell=0,\cdots,k$, then by \eqref{eq-vk-v0} and \eqref{eq-u0-v0w0}, we know 
	\begin{align*}
		&\norm{v_{k+1}}_{2} \\
		&\leq \norm{v_{0}}_{2}+ C_{0} \sum_{i=0}^{k} \brk{T_{i}^{E}\norm{v_{i}}_{2}^{\beta}+T_{i}^{F}\left((kN^{-\alpha\beta/A})^{-s}N^{-1}\right)^{\beta}} (kN^{-\alpha\beta/A})^{-s}N^{-1},\\
		&\leq N^{\alpha}+c_{0} \sum_{i=0}^{k} \brk{N^{-\alpha\beta E/A} N^{\alpha\beta}+N^{-\alpha\beta F/A} \left((kN^{-\alpha\beta/A})^{-s}N^{-1}\right)^{\beta}}(kN^{-\alpha\beta/A})^{-s}N^{-1}.
	\end{align*}
	So, we only need to choose $c_{0}\ll1$ and $\lambda>0$, such that 
	\begin{align*}
		\Cas{N^{-\alpha\beta E/A} N^{\alpha\beta} \sum_{i=0}^{k}(kN^{-\alpha\beta/A})^{-s}N^{-1} \lesssim N^{\alpha}, \\
		N^{-\alpha\beta F/A } \sum_{i=0}^{k}\brk{(kN^{-\alpha\beta/A})^{-s}N^{-1}}^{\beta+1}  \lesssim N^{\alpha},}
	\end{align*}
	hold for any $k=0,\cdots,K_{N}$.
	
	Combining these equations and \eqref{eq-Tk-condition}, we only need $\lambda$ to satisfy 
	\begin{align*}
		\Cas{-\dfrac{\alpha\beta D}{A} +\brk{\lambda(-s)+\dfrac{\lambda\beta s}{A}-1}^{\beta}\leq 0,\\
		-\dfrac{\alpha \beta E}{A}+\alpha\beta+\dfrac{\alpha\beta s}{A}-1+\lambda(1-s) \leq \alpha,\\
		-\dfrac{\alpha\beta F}{A}+(\beta+1)(\dfrac{\alpha\beta s}{A}-1) +\lambda(1-s(\beta+1)) \leq \alpha.} 
	\end{align*}
	After simplification, we have  
	\begin{align*}
		\Cas{\lambda(-s)  \leq \dfrac{\alpha D}{A}-\dfrac{\alpha\beta s}{A}+1, \\
		\lambda(1-s) \leq \alpha +\dfrac{\alpha\beta E}{A}-\alpha\beta-\dfrac{\alpha\beta s}{A}+1,\\
		\lambda(\dfrac{1}{\beta+1}-s) \leq \dfrac{\alpha}{\beta+1}+\dfrac{\alpha\beta F}{A(\beta+1)}-\dfrac{\alpha\beta s}{A}+1.}
	\end{align*}
	One can check that \begin{align*}
		\frac{\alpha D}{A}  = \alpha +\frac{\alpha\beta E}{A}-\alpha\beta=\frac{\alpha}{\beta+1}+\frac{\alpha\beta F}{A(\beta+1)}.
	\end{align*}
	So we only need to choose 
	\begin{align*}
		\lambda & =\frac{1}{1-s} \brk{\frac{\alpha D}{A}-\frac{\alpha\beta s}{A}+1}.
	\end{align*}
	Then we know that 
	\begin{align*}
		\sum_{i=0}^{K_{N}}T_{i}\approx N^{\lambda} N^{-\alpha\beta/A} \rightarrow \infty, \text{ as } N\rightarrow \infty,
	\end{align*}
	holds only when $\lambda>\alpha \beta/A$, which is equivalent to 
	\begin{align*}
		\alpha(\beta-D) & <A.
	\end{align*}
	Recall that $D=A+\dfrac{\beta\sigma}{2}-\beta\tilde{\alpha}=1-\beta\tilde{\alpha}+\dfrac{d-\gamma}{m}$. Then \begin{align*}
		\beta-D & =\beta(\tilde{\alpha}+1)-1-\frac{d-\gamma}{m}.
	\end{align*}
	So, we need that 
	\begin{align*}
		\alpha<\alpha_{\max}:= \Cas{\infty, & $\beta(\tilde{\alpha}+1)\leq 1+\dfrac{d-\gamma}{m}$;\\
		\dfrac{A}{\beta(\tilde{\alpha}+1)-1-\dfrac{d-\gamma}{m}}, &$\beta(\tilde{\alpha}+1)> 1+\dfrac{d-\gamma}{m}.$}
	\end{align*}
	By $\theta=\alpha/(1+\alpha)$, we could obtain the equation of $\theta_{\max}$ as shown in Theorem \ref{thm-GWP}
	
\end{proof}

\begin{proof}[Proof of Corollary \ref{cor-Mpp'}]
	For any $(1/r,1/q) \in \Omega_{0}$, using the fact  $M_{r,r^{\prime}} \hookrightarrow L^{r}$(Lemma \ref{lem-embed-mpq-wpq-Lp}) and H\"older's inequality, we have 
	\begin{align}\label{eq-St-embed-Mrr'}
		\norm{S_{m}(T+t)u}_{L_{I_{1}}^{q}L_{x}^{r}} \leq \norm{S_{m}(T+t)u}_{L_{I_{1}}^{\infty}L_{x}^{r}} \lesssim \norm{S_{m}(T+t)u}_{L_{I_{1}}^{\infty}(M_{r,r^{\prime}})}.
	\end{align}
	Denote $s_{p}=(m-2)(1/2-1/p)$. By Lemma \ref{lem-mpqs-sm(t)}, we know that 
	\begin{align*}
		\norm{S_{m}(T+t)u}_{M_{r,r^{\prime}}} \lesssim \inner{T}^{1/2-1/r} \norm{u}_{M_{r,r^{\prime}}^{s_{r}}}.
	\end{align*}
	Substituting this estimate into \eqref{eq-St-embed-Mrr'}, we have 
	\begin{align*}
		\norm{S_{m}(T+t)u}_{L_{I_{1}}^{q}L_{x}^{r}} \lesssim \inner{T}^{1/2-1/r} \norm{u}_{M_{r,r^{\prime}}^{s_{r}}}.
	\end{align*}
	By Definition \ref{def-Xqr-sigma}, we have $M_{r,r^{\prime}}^{s_{r}}\hookrightarrow X_{q,r}^{1/r-1/2,m}$ for any $q\geq 1$.
	
	By Theorem \ref{thm-GWP}, we know that \eqref{eq-fnlsh} is globally well-posed in $\brk{L^{2}, M_{r,r^{\prime}}^{s_{r}}}_{\theta,\infty}$, for any $(1/r,1/q)\in \Omega_{\gamma,\sigma} $ and $0<\theta<\theta_{\max}$.

	By the relation between real and complex interpolation (Theorem 3.9.1 in \cite{Bergh1976Interpolation}) and Lemma \ref{lem-mpq-complex-interpolation}, we know that $$M_{p,p^{\prime}}^{s_{p}} = \brk{L^{2}, M_{r,r^{\prime}}^{s_{r}}}_{[\theta]} \hookrightarrow \brk{L^{2}, M^{s_{r}}_{r,r^{\prime}}}_{\theta,\infty},$$ where 
	\begin{align*}
		\frac{1}{p} &=\frac{1-\theta}{2}+\frac{\theta}{r}=\frac{1}{2}-\theta(\frac{1}{2}-\frac{1}{r})>\frac{1}{2}-\theta_{\max}(\frac{1}{2}-\frac{1}{r}).
	\end{align*}
	Denote $\tilde{\alpha}_{0}=\dfrac{1+\dfrac{d-\gamma}{m}}{\beta}-1$. Then $\beta(\tilde{\alpha}+1)\leq 1+\dfrac{d-\gamma}{m}$ is equivalent to $\tilde{\alpha}\leq \tilde{\alpha}_{0}$. By Theorem \ref{thm-GWP}, we have 
	\begin{align*}
		x_{0} \leq \frac{1}{p_{\max}} = \inf_{\tilde{\alpha}\leq \tilde{\alpha}_{0}}\set{\frac{1}{r_{\max}}} \wedge \inf_{\tilde{\alpha}>\tilde{\alpha}_{0}}\set{\frac{1}{2}-\frac{A}{\beta}  \frac{\frac{1}{2}-\frac{1}{r}}{\tilde{\alpha}+1-\frac{d}{2m}} } ,
	\end{align*}
	where $(1/r,1/q)\in \Omega_{\gamma,\sigma}, \tilde{\alpha}=1/q+\sigma/r$. Notice that when $(x_{0},y_{0})\in \overline{\Omega}_{\gamma,\sigma}$, we have $\tilde{\alpha}(x_{0},y_{0})=\sigma/2$. So, when $\tilde{\alpha}_{0}\geq \sigma/2$, we know $1/p_{max}=x_{0}$. When $\tilde{\alpha}_{0}>\sigma/2$, we should calculate $$\inf_{\tilde{\alpha}\leq \tilde{\alpha}_{0}}\set{\frac{1}{r_{\max}}}, \   \inf_{\tilde{\alpha}>\tilde{\alpha}_{0}}\set{\frac{1}{2}-\frac{A}{\beta}  \frac{\frac{1}{2}-\frac{1}{r}}{\tilde{\alpha}+1-\frac{d}{2m}} }$$ respectively and choose the small one to be $1/p_{\max}$. This calculation follows the standard calculus procedure, which we will omit.

\end{proof}
%

\begin{proof}[Proof of Corollary \ref{coro-local-smooth}]
	When $m=2$, by local smoothing estimate in modulation space(Lemma \ref{lemma-local-smoothing}), for $2\leq p\leq p_{0}=2+4/d, s>0$, we have 
	\begin{align*}
		\norm{S_{m}(t)u}_{L_{t\in I_{1}}^{p}L_{x}^{p}} & \lesssim \norm{u}_{M_{p,2}^{s}}.
	\end{align*}
	 Denote $\lambda=1+T, u_{\lambda}(x)=u(\lambda x)$. By scaling, Lemma \ref{lemma-scaling-mpq}, and the estimate above, we have 
	\begin{align*}
		\norm{S_{m}(t+T)u}_{L_{t\in I_{1}}^{p}L_{x}^{p}} & =\norm{S_{m}(t)u}_{L_{t\in [T,T+1]}^{p}L_{x}^{p}}  \leq \norm{S_{m}(t)u}_{L_{t\in [0,T+1]}^{p}L_{x}^{p}}\\
		&=\norm{S_{m}((T+1)t)u}_{L_{t\in [0,1]}^{p}L_{x}^{p}}=\norm{S_{m}(t)u_{\lambda^{1/m}}(\lambda^{-1/m}x)}_{L_{t\in [0,1]}^{p}L_{x}^{p}}\\
		 &=\lambda^{\frac{d}{m p}} \norm{S_{m}(t)u_{\lambda^{1/m}}}_{L_{t\in [0,1]}^{p}L_{x}^{p}}\lesssim \lambda^{\frac{d}{m p}} \norm{u_{\lambda^{1/m}}}_{M_{p,2}^{s}}\\
		 &\lesssim \lambda^{\frac{1}{m}(s+d(\frac{1}{2}-\frac{1}{p}))} \norm{u}_{M_{p,2}^{s}}\approx \inner{T}^{\frac{1}{m}(s+d(\frac{1}{2}-\frac{1}{p}))} \norm{u}_{M_{p,2}^{s}},
	\end{align*}
	which means that \begin{align} \label{eq-Mp2-embed-Xpp}
		M_{p,2}^{s} \hookrightarrow X_{p,p}^{-\frac{1}{m}(s+d(\frac{1}{2}-\frac{1}{p})),m}, \ \forall \ 2\leq p\leq p_{0}, s>0.
	\end{align}
	\begin{itemize}
		\item when $\gamma>\dfrac{d}{2}, \tilde{\beta}>\dfrac{1}{\sigma}(3-\dfrac{2\gamma}{d})$: one can check that $x_{0}\vee x_{4}<\dfrac{1}{p_{0}}<x_{1}$. So we know that $(1/p_{0},1/p_{0})\in \Omega_{\gamma,\sigma}$.
		Choose $p=p_{0}$ in \eqref{eq-Mp2-embed-Xpp}. Then by Theorem \ref{thm-GWP}, we could obtain the global solution with $u_{0}\in \brk{L^{2},M_{p_{0},2}^{0+}}_{[\theta]}=M_{p,2}^{s}$, where
		\begin{align*}
			&s>0, \ \frac{1}{p}=\frac{1}{2}-\theta \brk{\frac{1}{2}-\frac{1}{p_{0}}},\\
			&0\leq \theta <\theta_{\max}=\Cas{1, & when $\beta(\dfrac{\sigma}{2}+1) \leq 1+\dfrac{d-\gamma}{m}$,\\
			\dfrac{A}{\beta}, & when $\beta(\dfrac{\sigma}{2}+1) >1+\dfrac{d-\gamma}{m}$.}
		\end{align*} 
		By substituting $\theta_{\max}$, we can obtain the upper bound $p_{\max}$ as desired in the corollary.
		\item when $\gamma>\dfrac{d}{2}, \tilde{\beta}\leq\dfrac{1}{\sigma}(3-\dfrac{2\gamma}{d})$ or $\gamma\leq \dfrac{d}{2}$: we have $\dfrac{1}{p_{0}}<x_{0}$. To ensure that the line segment $(t,t), t\in[1/p_{0},1/2]$ intersects with $\Omega_{\gamma,\sigma}$, we need the point $(x_{1},y_{1})$ to be located above this line segment, which means that $y_{1}>x_{1}$. By calculation, this is equivalent to $\tilde{\beta}>\dfrac{1}{\sigma}\brk{2-\dfrac{\gamma}{d}}$, which is just the assumption in Corollary \ref{coro-local-smooth}. 
		
		To obtain the upper bound $p_{\max}$, we choose $\dfrac{1}{p}=\dfrac{\sigma}{\sigma \beta-1}(1-\dfrac{\gamma}{d})$ in \eqref{eq-Mp2-embed-Xpp}, which is the  $x$-coordinate of the intersection point between the line segment $(t,t), t\in[1/p_{0},1/2]$ and the boundary of the upper left corner of $\Omega_{\gamma,\sigma}$. Notice that in the case, we always have $\beta(\tilde{\alpha}+1)>1+\dfrac{d-\gamma}{m}$. Then by Theorem \ref{thm-GWP}, we can obtain the upper bound $p_{\max}$ as desired in the corollary.
	\end{itemize}
\end{proof}

\section*{Acknowledgments}

The author would like to thank Divyang G. Bhimani for help discussions.
This work was paritally supported by the National Natural Science Foundation of China (Grant No. 12401119) and Natural Science Foundation of Fujian Province (Grant No. 2024J08068).


\end{document}